\documentclass{article}

\usepackage{amssymb}
\usepackage{amsmath}
\usepackage{amsfonts}
\usepackage{amsthm}
\usepackage{mathabx}
\usepackage{graphicx}
\usepackage{cite}
\usepackage[breaklinks=true]{hyperref}
\usepackage{graphicx}
\usepackage{tikz}
\usepackage{placeins}
\usepackage{float}
\usepackage{pstricks}
\usepackage{tikz-cd}
\usepackage{todonotes}
\usepackage{cleveref}

\DeclareMathOperator{\CAT}{CAT}

\usetikzlibrary{arrows,chains,matrix,positioning,scopes}

\newtheorem{thm}{Theorem}[section]
\newtheorem{theorem}[thm]{Theorem}
\newtheorem{lem}[thm]{Lemma}
\newtheorem{lemma}[thm]{Lemma}
\newtheorem{prop}[thm]{Proposition}
\newtheorem{cor}[thm]{Corollary}
\newtheorem{conj}[thm]{Conjecture}

\newtheorem*{thmex}{Theorem \ref{thm:excessive}}
\newtheorem*{corfree}{Corollary \ref{cor:freebyfree}} 
\newtheorem*{thmcube}{Theorem \ref{thm:cube}} 
\newtheorem*{thmf2}{Theorem \ref{thm:f2fn}}
\newtheorem*{thmfiber}{Theorem \ref{thm:strongfiber}}

\theoremstyle{definition}
\newtheorem{definition}[thm]{Definition}

\theoremstyle{remark}

\newtheorem*{prop proof}{Proof of Proposition}

\newcommand{\RR}{\mathbb{R}}      
\newcommand{\ZZ}{\mathbb{Z}}        
\newcommand{\QQ}{\mathbb{Q}}      

\newcommand\restr[2]{{
		\left.\kern-\nulldelimiterspace 
		#1 
		\vphantom{\big|} 
		\right|_{#2} 
}}

\newcommand{\Aut}{\mathrm{Aut}} 
\newcommand{\Out}{\mathrm{Out}}
\newcommand{\Isom}{\mathrm{Isom}}

\numberwithin{equation}{section}

\author{Robert Kropholler and Genevieve Walsh}
\date{}
\title{Incoherence and fibering of many free-by-free groups} 
\begin{document}
	\maketitle
	\begin{abstract}
		We show that free-by-free groups satisfying a homological criterion, which we call excessive homology, are incoherent. This class is large in nature, including many examples of hyperbolic and non-hyperbolic free-by-free groups. We apply this criterion to finite index subgroups of $F_2\rtimes F_n$ to show incoherence of all such groups, and to other similar classes of groups. Furthermore, we show that a large class of groups, including free-by-free, surface-by-surface, and finitely generated-by-RAAG,  algebraically fiber if they have excessive homology.  
	\end{abstract}
	
	\section{Introduction}\label{sec:intro}
	
	We begin with a definition. 
	
	\begin{definition} A group is {\it coherent} if every finitely generated subgroup is finitely presented.  \end{definition} 
	A group is $G$ is called {\it incoherent} if $G$ has a  finitely generated subgroup $H$ which is not finitely presented, and we call such a subgroup a {\it witness to incoherence}. There are many examples of incoherent groups.  For example, $F_2 \times F_2$ is well known to be incoherent. A construction of many incoherent groups is given by Rips \cite{Ripsconstruction}. Here we present some substantial evidence towards the following conjecture, which was independently and previously made by Dani Wise: 
	
	\begin{conj} Let $ G = F_m \rtimes F_n$, where $m, n \geq 2$.  Then $G$ is incoherent. \end{conj} 
	
	In this paper we show that this conjecture is true when $G = F_m \rtimes F_n$  and $rk(H^1(G; \RR)) \geq n+1$ in Theorem \ref{thm:excessive} and for $G = F_2 \rtimes F_n$ in Theorem \ref{thm:f2fn}. Note that in contrast, $F_m \rtimes \mathbb{Z}$ is always coherent \cite{FeignHandel}.  
	
	The techniques of this paper are inspired by two examples. The first is Bowditch and Mess's example of an incoherent hyperbolic 4-manifold group, \cite{BowditchMess}. To construct this example one starts with a closed 3-manifold $M$ which contains a totally geodesic surface $S_g$, such that $M$ is also fibered with fiber $F$. They then consider the following space $M\cup_{S_g}M$. This space is homotopy equivalent to a compact quotient of a convex subset of $\mathbb{H}^4$ by a subgroup of $\Isom(\mathbb{H}^4)$.  This is a higher-dimensional analog of gluing two closed surfaces together along a geodesic and thickening to get a convex co-compact Kleinian manifold. They show, among other things, that the resulting hyperbolic 4-manifold has incoherent fundamental group. The witness to incoherence is obtained by taking the subgroup generated by the two fibers. This subgroup can be written as an amalgamated free product as $F\ast_{F\cap S_g}F$. 
	
	The second example we present here is a proof that $F_2\times F_2$ is incoherent. We can write $F_2\times F_2$ as an amalgamated free product as $(F_2\times \ZZ)\ast_{F_2}(F_2\times \ZZ) = \langle a, b, s\rangle\ast_{\langle a, b\rangle}\langle a, b, t\rangle$. We can take a non-standard fibration of each side and get fibers $\langle as^{-1}, bs^{-1}\rangle$ and $\langle at^{-1}, bt^{-1}\rangle$. The union of these two fibers gives a witness to incoherence which can be written as an amalgamated free product $\langle as^{-1}, bs^{-1}\rangle \ast_{\langle a, b\rangle\cap \langle at^{-1}, bt^{-1}\rangle}\langle at^{-1}, bt^{-1}\rangle$. 
	
	The analogs to this construction here are Corollaries \ref{cor:fibering} and \ref{cor:freebyfree}.  
	
	One should note that in both examples we have ommited a key detail. Namely, that the amalgamating subgroups $F\cap S_g$ and $\langle a, b\rangle\cap \langle at^{-1}, bt^{-1}\rangle$ are not finitely generated. This is because, as discussed later, free groups and surface groups do not algebraically fiber. 
	
	Our results are quite general and apply to constructions involving finitely generated groups which do not algebraically fiber.  Most of our results follow from our main theorem.  If $G =  H \rtimes F_k$, then we say $G$ has {\em excessive homology} if $rk(H^1(G; \mathbb{R})) \geq k+1$. 
	
	\begin{thmex}  Let $G = H \rtimes F_k$, where $H$ is finitely generated and does not algebraically fiber and $k\geq 2$.   If $G$ has excessive homology, then $G$ is incoherent. \end{thmex}

	Some consequences are as follows.  
	\begin{corfree}  Let $G = F_m \rtimes F_n $. If $G$ has excessive homology, then $G$ is incoherent. \end{corfree} 
	A group is {\it virtually special} if it virtually acts co-specially on a $\CAT(0)$ cube complex. 
	\begin{thmcube} Let $G = H \rtimes F_n$, $n \geq 2$  where $H$ is either a closed hyperbolic surface group or a free group of rank $\geq 2$.   If $G$ is hyperbolic and virtually special,  then $G$ is incoherent. \end{thmcube} 
	
	Using the techniques of the previous theorems, we can show that many groups algebraically fiber, such as groups of the form $\pi_1(S_g) \rtimes \pi_1(S_h)$, which have excessive homology. 
	
	\begin{thmfiber} Suppose that $G$ fits into a short exact sequence  $$ 1 \rightarrow H \rightarrow G \rightarrow Q \rightarrow 1$$  where $H$ is finitely generated, $Q$ is generated by $\lbrace x_1,..,x_n \rbrace$, and $rk H^1(Q, \RR) =n$.  If $rkH^1(G; \RR) >n$, then $G$ algebraically fibers. 
	\end{thmfiber} 

As an application of these results, we directly find a finite index subgroup that satisfies our homological condition and use this to show: 

	\begin{thmf2} \label{thm:F2} Let $G = F_2 \rtimes F_n$.  Then $G$ is incoherent.  Moreover, $G$ virtually algebraically fibers. \end{thmf2} 
	
	The general plan of the paper is as follows. In Section \ref{sec:fiber} we discuss algebraic fiberings and reveiw some results of Bieri-Neumann-Strebel that we will use. In Section \ref{sec:coherence} we make some general remarks on coherence. Our main Theorem \ref{thm:excessive} is proven in Section \ref{sec:Hbyfree} as well as several corollaries and related results including Corollary \ref{cor:freebyfree} and Theorem \ref{thm:cube}.  In Section 5, we study fibering for a large class of groups including free-by-free and prove Theorem \ref{thm:strongfiber}.  Theorem \ref{thm:f2fn} is proven in Section \ref{sec:F2}. 
	\vskip .2 in 
	{\bf Acknowledgements:}  We thank the other participants of the Emerging Topics Workshop: ``Coherence and Quasiconvex subgroups" held at The Institute for Advanced Study in March 2019 for productive conversations. We also thank Stefano Vidussi for comments on an earlier version of the paper.   The second author was partially supported through NSF DMS- 1709964.

	\section{Background on algebraic fiberings} \label{sec:fiber} 
	\begin{definition}
		We say that a group $G$ {\em algebraically fibers} if there exists a map to $\ZZ$ with finitely generated kernel. 
	\end{definition}
	
	There are many examples of groups that do not fiber. For instance, $F_n, n\geq 2$, $\pi_1(S_g), g\geq2$ and $BS(1, n)$.  
	
	The set of algebraic fibers of a group is analogous to the fibers of a hyperbolic 3-manifold, and there is a well-developed theory of algebraic fibers (which is part of a more comprehensive theory) similar to the theory of fibrations \cite{Thurstonnorm}.

	\begin{definition}
		The {\em character sphere of $G$}, $S(G)$, is $H^1(G; \RR)\smallsetminus \{0\}/\sim$, where $\chi\sim\chi'$ if there is $\lambda\in \RR_+$ such that $\chi = \lambda\chi'$. These are equivalence classes of maps $G \rightarrow \RR$, and we call elements of $S(G)$ {\em characters.}
	\end{definition}
	
	Bieri, Neumann and Strebel described the following invariant in \cite{BNS87}.  Also see \cite{Strebelnotes}.  We follow notation and give specific references from these notes. 
	
	\begin{definition}
		The {\em BNS invariant} $\Sigma^1(G)$ is the set of characters $\chi \in S(G)$ such that the full subgraph on the vertices of $\mathrm{Cay}(G, S)$ where $\chi(v)\geq0$ is a connected graph. 
	\end{definition}
	
	\begin{theorem}\cite[Corollary A4.3]{Strebelnotes}
		Let $\chi\in S(G)$. Then $\ker(\chi)$ is finitely generated if and only if $\chi, -\chi\in \Sigma^1(G)$. 
	\end{theorem}
	
	\begin{theorem} \label{thm:bnsopen}\cite[Corollary A3.3]{Strebelnotes}
		$\Sigma^1(G)$ is an open subset of $S(G)$. 
	\end{theorem}
	
	\begin{cor}
		Suppose that $rk(H^1(G;\RR))\geq 2$ and there exists $\chi$ with finitely generated kernel. Then there are infinitely many other $\chi'\colon G\to\ZZ$ such that $\ker(\chi')$ is finitely generated. 
	\end{cor}
	
	\section{Background on incoherence} \label{sec:coherence}

	While many groups are known to be incoherent by a result of Rips \cite{Ripsconstruction}, the most important concrete example is $F_2 \times F_2$.  The original proof is attributed to Stallings and uses the second homology of the group.  The proof in Section \ref{sec:intro} is related to our techniques. We provide the key detail missing from Section \ref{sec:intro} here. To create a finitely generated but not finitely presented group we use the following lemma from \cite{BowditchMess} where it is attributed to B. Neumann.  
	
	\begin{theorem}\label{notfinpresented}\cite{BowditchMess} (B. Neumann) 
		Let $G_1, G_2$ be finitely generated groups. Let $G = G_1\ast_H G_2$. If $H$ is not finitely generated, then $G$ is not finitely presented.
	\end{theorem}
	
	Now consider $F_2 \times F_2 = \langle a,b\rangle \times \langle s,t\rangle$.  Consider the map $\phi$ to $\mathbb{Z}$ that sends each generator to 1.  Then let  $ H= \langle as^{-1}, bs^{-1}\rangle *_K \langle at^{-1},bt^{-1}\rangle$, where $K$ is the kernel restricted of $\phi$ to $\langle a,b\rangle$.  Since free groups do not algebraically fiber, $K$ must be infinitely generated and the incoherence of $F_2 \times F_2$ follows from Neumann's theorem. 
	
	Since any group that contains an incoherent group is incoherent, we have immediately that the right-angled Coxeter group on the graph $K_{3,3}$ is incoherent. However, to illustrate the subtlety of the problem, we note that if even one edge is subdivided, this group is coherent. To see this we use the following result of Karass and Solitar.
	
	\begin{theorem} \cite{KarrassSolitar}
		Let $G, G'$ be coherent groups. Let $H$ be a subgroup of $G$ and $G'$. If every subgroup of $H$ is finitely generated, then $G\ast_H G'$ is coherent. 
	\end{theorem}
	
	Therefore, if we even subdivide one edge of the $K_{3,3}$ graph, the resulting right angled Coxeter group is coherent.  Indeed we can write this new group  as a free product with amalgamation of two groups on planar graphs over a virtually cyclic group.  Since a right-angled Coxeter group defined by a planar graph is virtually a 3-manifold group \cite[Theorem 11.4.1]{DavisOkun}, both of these groups are coherent.  Therefore, Karrass and Solitar's result implies that the right-angled Coxeter group on the subdivided graph is coherent. 
	
	More generally, the same argument shows that: 
	
	\begin{prop} The right-angled Coxeter group on the barycentric subdivision of any graph is coherent.  \end{prop}

	\section[Incoherence of finitely generated-by-free groups]{Incoherence of $H$-by-free groups} \label{sec:Hbyfree} 
	
	We will be studying groups of the form $H\rtimes F_k$. It will be useful to first have an understanding on the homology of such a group.
	
	\begin{lemma}\label{lem:homologyoffreebyfree}
		Let $G = H\rtimes F_k$ be a finitely presented group. Let $\phi_1, \dots, \phi_k$ be the corresponding automorphisms. Let $\Phi_i$ be the automorphism induced on the abelianisation of $H$. Then $H_1(G;\ZZ) = \ZZ^k\times (H_1(H;\ZZ)/\langle(\Phi_i - I)(H_1(H;\ZZ))\rangle)$, here $I$ is the identity matrix.
	\end{lemma}
	\begin{proof}
		Let $H = \langle a_1, \dots, a_n\mid R\rangle.$ There is a presentation for $G$ of the form $$\langle a_1, \dots, a_n, t_1, \dots, t_k\mid t_ia_jt_i^{-1} = \phi_i(a_j), R\rangle.$$ Thus, when we abelianise we arrive at $\ZZ^{k}\oplus H_1(H;\ZZ)$ with the extra relations that $\phi_i(a_j) = a_j$. We replace $\phi_i$ with the map on the abelianization $\Phi_i$ and rewrite the relation as $\Phi_i(a_j)-a_j$ or $(\Phi_i - I)(a_j)$. Thus, we arrive at the desired conclusion. 
	\end{proof}
	
	The proof that $H^1(G;\RR) = \RR^k\times (H^1(H;\RR))/\langle(\Phi_i - I)(H^1(H;\RR))\rangle)$ is extremely similar. 
	\vskip .2 in

	Given an extension $H\rtimes F_k$ there is a natural map $F_k\to \mathrm{Aut}(H)$. This also gives a natural map $F_k\to\mathrm{Out}(H)$. 
	Our main interest is when $H$ contains a non-abelian free subgroup. In this case we can reduce to the case that the map $F_k\to\mathrm{Out}(H)$ is injective. 
	
	\begin{lem}
		Let $G = H\rtimes F_k, k\geq 2$. Suppose that $H$ contains a non-abelian free subgroup and that the natural map $F_k\to\mathrm{Out}(H)$ is not injective. Then $G$ contains a copy of $F_2\times F_2$. Thus, $G$ is incoherent. 
	\end{lem}
	\begin{proof}
		Since the map $F_k\to\mathrm{Out}(H)$ is not injective let $s$ be a non-trivial element of the kernel and let $t$ be a conjugate of $s$ in $F_k$ such that $s$, $t$ generate a free subgroup.  Let $a, b\in H$ be generators of a free subgroup of $H$. Then $s$ and $t$ both act by conjugation on $H$, that is, $s^{-1}as = g_sag_s^{-1}, s^{-1}bs = g_sbg_s^{-1}, t^{-1}at = g_tag_t^{-1}$ and $t^{-1}bt = g_tbg_t^{-1}$, for some $g_s$ and $g_t$ in $H$.  Thus we can see that $\langle a, b, sg_s, tg_t\rangle$ is a copy of $F_2\times F_2$. 
	\end{proof}
	
	Thus for the most part we will be interested in cases where the map $F_k\to\mathrm{Out}(H)$ is injective. 
	
	In general, we will consider the case that $H$ does not  algebraically fiber. 
	
	\begin{lemma} \label{lem:construction} 
		Let $G_i = H\rtimes_{\phi_i}\ZZ$, where $H$ does not algebraically fiber. Suppose that $\alpha_i\colon G_i\to\ZZ$ are homomorphisms such that $\restr{\alpha_1}{H} = \restr{\alpha_2}{H}$ are non trivial. Suppose further that $K_i = \ker(\alpha_i)$ is finitely generated. Then $G = G_1\ast_H G_2$ is incoherent. 
	\end{lemma}
	\begin{proof}
		We know that $K_i\cap H$ is also the kernel of a homomorphism to $\ZZ$ and so is infinitely generated since $H$ does not algebraically fiber. Furthermore since $\restr{\alpha_1}{H} = \restr{\alpha_2}{H}$ we know that $K_1\cap H = K_2\cap H = L$. Let $N$ be the subgroup of $G$ generated by $K_1$ and $K_2$. We can write $N$ as an amalgamated free product $K_1\ast_{L}K_2$. This is the amalgamated free product of two finitely generated groups over an infinitely generated group and so is not finitely presented by Theorem \ref{notfinpresented}.
	\end{proof}
	
	\begin{theorem}\label{thm:incohwithhomology}
		Let $G_i = H\rtimes_{\phi_i}\ZZ$. Assume that $H$ is finitely generated and does not algebraically fiber. Let $G = G_1\ast_HG_2$. If $H^1(G; \RR)$ has rank $\geq 3$, then $G$ is incoherent. 
	\end{theorem}
	\begin{proof}
		Let $G_1 = \langle H, s\rangle$ and $G_2 = \langle H, t\rangle$. Let $\alpha_s$ be the homomorphism $G\to \ZZ$ be defined by counting the exponent sum of $s$, define $\alpha_t$ similarly. Since $H^1(G;\RR)$ has rank $\geq 3$, there is another class $\gamma\colon G\to \ZZ$ which is not in the span of $\alpha_s$ and $\alpha_t$. 
		
		We now use the BNS invariant $\Sigma^1(G_i)$ to find other fiberings of $G_i= H \rtimes_{\phi_i} \ZZ$. Note that $\Sigma^1(G_i)$ is not empty since $H$ is finitely generated. 
		
		Consider the homomorphisms $\beta_1 = a\gamma|_{G_1} + \alpha_s|_{G_1}$ and $\beta_2 = a\gamma|_{G_2} + \alpha_t|_{G_2}$. These define two homomorphisms from $G_i\to\RR$. If we pick $a$ to be rational we can assume that the images of $\beta_1, \beta_2$ are cyclic subgroups. 
		
		Since the BNS invariant is an open subset of the character sphere, we can take $a$ small enough so that the kernels of $\beta_1$ and $\beta_2$, respectively, are finitely generated subgroups $K_1$ and $K_2$ of $G_1$ and $G_2$ respectively.
		
		By construction $K_1\cap H = \ker(\restr{\beta_1}{H}) = \ker(\restr{a\gamma}{H}) = \ker(\restr{\beta_2}{H}) = K_2\cap H$. Since $H$ doesn't algebraically fiber $K_1\cap H$ is not finitely generated.  Thus we are in the situation of Lemma \ref{lem:construction} and there is an finitely generated infinitely presented subgroup of $G$.  
	\end{proof}

	Note here that $G$ can also be written as $H \rtimes F_2$, and we are requiring that $G$ has one more map to $\ZZ$ than comes from the natural map to $F_2$.   Recall that if $G = H \rtimes F_k$, we say that $G$ has {\it excessive homology} if $rk (H_1(G; \RR )) \geq k+1$.

	\begin{theorem} Let $G = H \rtimes Q$ where $rk(H^1(G;\mathbb{R})) > rkH^1(Q; \mathbb{R})$, where $H$ is finitely generated and does not algebraically fiber and $Q$ contains a non-abelian free subgroup.   Then $G$ is incoherent. \end{theorem} 

\begin{proof} The condition that $rk(H^1(G;\mathbb{R}) > rkH^1(Q; \mathbb{R})$ can be rephrased as there is a map $\gamma$ from $G$ to $\ZZ$ such that an element of $H$ has non-trivial image. Consider the subgroup of $G$ given by $N = H\rtimes F_2$ where the $F_2$ is a non-abelian free subgroup of $Q$. We see that $N$ has excessive homology from the restriction of $\gamma$ to $N$. We can now appeal to Theorem \ref{thm:incohwithhomology} deducing that $N$ is incoherent and hence so is $G$. 
	\end{proof}

	We have the following immediate corollary.  
	
	\begin{theorem} \label{thm:excessive}  Let $G = H \rtimes F_k$, where $H$ is finitely generated and does not algebraically fiber.   If $G$ has excessive homology, then $G$ is incoherent. \end{theorem}

	We now detail some consequences of Theorems \ref{thm:incohwithhomology} and \ref{thm:excessive}. 
	
	\begin{cor} \label{cor:fibering} 
		Let $M_1, M_2$ be two fibered 3-manifolds with isomorphic fiber $S_g, g\geq 2$. Let $X = M_1\cup _{S_g}M_2$. Suppose that $H^1(X;\RR)$ has rank $\geq 3$. Then $\pi_1(X)$ is incoherent. 
	\end{cor}
	\begin{proof}
		As noted previously fundamental groups of surfaces do not algebraically fiber. Thus we can apply Theorem \ref{thm:incohwithhomology}.
	\end{proof}
	
	\begin{cor} \label{cor:freebyfree} Let $G = F_m \rtimes F_n$, $m, n\geq 2$.  If  $H^1(X;\RR)$ has rank $\geq n +1$, then $G$ is incoherent. \end{cor}

	\begin{thm} \label{thm:cube} Let $G = H \rtimes F_n$, $n \geq 2$  where $H$ is either a closed hyperbolic surface group or a free group of rank $\geq 2$.   If $G$ is hyperbolic and virtually special,  then $G$ is incoherent.  \end{thm} 
	
	\begin{proof} Since $G$ is virtually special, it is virtually a subgroup of a right-angled Artin group  by \cite{HaglundWise}.  Such groups virtually retract onto their quasi-convex subgroups \cite[Theorem F]{Haglund}. An infinite cyclic subgroup of a hyperbolic group is quasi-convex. Now consider a virtual retraction onto some infinite cyclic subgroup of $H < H \rtimes F_n$.  This is a map from $G' = H' \rtimes F_l  \rightarrow \mathbb{Z}$ which illustrates that $G'$ has excessive homology, since it is not a linear combination of the maps which are retracts to cyclic subgroups of $F_l$.  Then by Theorem \ref{thm:excessive}, $G'$ is incoherent and so is $G$. \end{proof} 
	
	In the case that $G$ above is hyperbolic, cubulation is  sufficient hypothesis. Indeed, if a group is hyperbolic and $\CAT(0)$ cubulated, it virtually acts co-specially on a $\CAT(0)$ cube complex by \cite{Agol}.  By the work of Gersten \cite{Ger}  one can show that there are $F_m \rtimes F_n$ groups with $m,n \geq 2$ which are not $\CAT(0)$.
	
	\begin{cor}
		Suppose that $G = H\rtimes F_k$ where $H^1(H;\RR) = \RR$. Suppose further that $H$ is finitely generated and does not algebraically fiber. Then $G$ is incoherent. 
	\end{cor}
	
	\begin{proof} The action of $F_k$ on $H^1(H;\RR)$ will leave the one-dimensional vector space  $H^1(H; \RR)$ invariant.  There is an index two subgroup of $G$ such that the induced action of an index 2 subgroup of $F_k$ is the identity on $H^1(H;\RR)$.  Thus there is an index 2 subgroup with excessive homology by Lemma \ref{lem:homologyoffreebyfree}.  Then by Theorem \ref{thm:excessive}, $G$ is incoherent. 
	\end{proof} 
	For example, if $|m|\neq |n|$, then $G = BS(n, m)\rtimes F_k$ is incoherent since $BS(n,m)$ has one-dimensional 1st cohomology.

	\section[Fibering of finitely generated-by-free groups]{Fibering of $H$-by-free and related groups}

	Note that there are a lot of different ways to get witnesses to incoherence.  For example, we could have followed the proof of Theorem \ref{thm:incohwithhomology} replacing 2 maps by $k$ maps.   In addition, with some care as to the choice of map $\gamma$ in the proof of Theorem \ref{thm:incohwithhomology} we can arrange that $G$ virtually algebraically fibers. We are stating this as a separate theorem, since neither directly implies the other.  However, if $G$ has cohomological dimension 2 and non-zero Euler characteristic, then a result of Bieri \cite{Bieri} combined with Theorem \ref{thm:vf} will show that $G$ is incoherent. 
	
	\begin{theorem} \label{thm:vf} Let $G = H \rtimes F_k$.  If $H$ is finitely generated and $G$ has excessive homology, then $G$ algebraically fibers. \end{theorem} 
	
	\begin{proof}
		As noted above, excessive homology is equivalent to finding a homomorphism $\gamma\colon G\to \ZZ$ which has non-trivial image when restricted to $H$. Let $a_1, \dots, a_k$ be a generating set for $F_k$. Let $\alpha_i\colon G\to \ZZ$ be the exponent sum of $a_i$. By taking appropriate linear combinations of $\gamma$ and $\alpha_i$ we will assume that $\gamma(a_i) = 0$ for $a_1, \dots, a_k$ and that $\gamma|_H$ is surjective. 
		
		Let $G_i = H\rtimes \langle a_i\rangle$. 
		Let $r$ be an integer to be fixed later. Consider the homomorphism $\beta_i = \alpha_i + \frac{1}{r}\gamma\colon G_i\to \QQ$.  Since the image is finitely generated it is cyclic. Assume $r$ is large enough so that $\ker(\beta_i)$ is finitely generated by Theorem \ref{thm:bnsopen}. Since $\frac{1}{r}\neq 0$ we see that $\beta_i|_H$ is non-zero. By picking $r$ large enough we can ensure that $\beta_i|_H$ is surjective onto the image of $\beta_i$. 
		
		Let $N = \langle \cup_{i=1}^k\ker(\beta_i)\rangle$. This is a finitely generated subgroup of $G$. One should note that this is the subgroup constructed in the proof of Theorem \ref{thm:incohwithhomology}. We will show that $N$ is normal and $G/N$ is infinite cyclic. 
		
		To see that $N$ is normal let $g\in G_i$ and $w\in \ker(\beta_j)$ and consider $gwg^{-1}$. Since $\beta_i|_H$ is surjective we can find $h\in H$ such that $gh\in N$. Note that $gwg^{-1} = ghh^{-1}whh^{-1}g^{-1}$, and since $gh\in N$, we are left to prove that $h^{-1}wh\in N$. Since $H\triangleleft G_i$ for each $i$ and $\ker(\beta_j)\triangleleft G_j$, we can conclude that $h^{-1}wh\in \ker(\beta_j)$ and that $N\triangleleft G$. 
		
		Since $\beta_i|_H$ is surjective we see that every element of $G/N$ is equivalent to an element in $H/(H\cap N)$ which by the first isomorphism theorem is $\ZZ$.  Thus $G$ is virtually fibered. 
	\end{proof}

	In fact, excessive homology is exactly the condition needed for fibering and we obtain the following. 
	
	\begin{theorem}
		Let $G = H\rtimes F_k$. Suppose that $H$ is finitely generated. Then $G$ virtually algebraically fibers if and only if $G$ has virtually excessive homology. 
	\end{theorem}
\begin{proof}
	Let $K\leq G$ be a finite index subgroup. Then we get a homomorphism $\pi\colon K\to F_k$ which has finite index image $F_l$. We can now write $K$ as $L\rtimes F_l$, where $L = K\cap H$ is a finite index subgroup of $H$ and hence finitely generated. 
	
	If $K$ has excessive homology, then we can appeal to Theorem \ref{thm:vf}. 
	
	For the other direction suppose that $K$ virtually fibers. Assume for a contradiction that $rk(H_1(K;\ZZ)) = l$. Then every homomorpshism $\phi\colon K\to \ZZ$ factors through $F_l$. Thus we get a surjection from $\ker(\phi)\to \ker(\phi')$, where $\phi = \phi'\circ\pi$.  The latter kernel is infinitely generated, since $F_l$ does not algebraically fiber.  Thus $\ker(\phi)$ is infinitely generated and $K$ cannot fiber. 
\end{proof}

In fact, we can use the subgroup in Theorems \ref{thm:incohwithhomology}, and \ref{thm:vf}, to find algebraic fibers in a wider class of groups. 
\begin{theorem} \label{thm:strongfiber}
	
	Suppose that $G$ fits into a short exact sequence  $$ 1 \rightarrow H \rightarrow G \rightarrow Q \rightarrow 1$$  where $H$ is finitely generated, $Q$ has rank $n$, and $rk( H^1(Q; \RR)) =n$.  If $rk(H^1(G; \RR)) >n$, then $G$ algebraically fibers. 
	\end{theorem} 
	 
	 \begin{proof} Denote the surjective map $G \rightarrow Q$ by $\phi$. Let $\lbrace y_1,...y_n \rbrace$ be elements of $G$ such that $\{\phi(y_i)\}$ is a generating set for $Q$.  Let $F_n$ be the free group $\langle t_1, ... t_n \rangle$.  Since $H$ is normal, $y_iH y_i^{-1}=H$, so each $y_i$ induces an automorphism of $H$.     Let $H \rtimes F_n$ be the semi-direct product where conjugation by $t_i$, induces the same automorphism of $H$ as $y_i$ acting on $H$ in $G$.  Then we have the following commutative diagram:  
	 
	\[
	 \begin{tikzcd}
  1 \arrow[r] & H \arrow[d] \arrow[r] & H \rtimes F_n \arrow[d] \arrow[r] & F_n\arrow[d, "f"] \arrow[r] & 1 \\
  1 \arrow[r] & H \arrow[r] & G \arrow[r, "\phi"] & Q \ar[r] & 1
\end{tikzcd}
	 \]
	
	 Now since $rk(H^1(G, \RR)) >n$, there is a map $\gamma\colon G \rightarrow \ZZ$ that does not factor through $Q$.  Since $rk (H^1(Q; \RR)) =n$, we can use linear combinations of maps to $\ZZ$ that do factor through $Q$ (where  $y_i$ maps to 1) to assure that $\gamma|H $ is onto and each $\gamma(y_i) =0$.  This is not necessary but it will make it easier to see the fiber. The map $\gamma\colon G \rightarrow \ZZ$ lifts to a map $\hat \gamma\colon H \rtimes F_n$ where $\hat \gamma|_H = \gamma|_H$ and $\hat \gamma(t_i) = \gamma(y_i) =0$.  Then, $H\rtimes F_n$ has excessive homology and by Theorem \ref{thm:vf}, we have that $H \rtimes F_n$ algebraically fibers.   Let $\hat p\colon H \rtimes F_n \rightarrow \ZZ$ be any algebraic fibration.   Since $rk (H^1(Q; \RR)) =n$, we can construct the map $p\colon G \rightarrow \ZZ$ defined by  $p|_H = \hat p|_H$ and $p(y_i) = \hat p(t_i)$.  The map $\hat p$ factors through $G$ and is equal to $p\circ f$, as can be seen by the definition on generators.   Then we have the following commutative diagram. 
	 
	 \[
	 \begin{tikzcd}
  1 \arrow[r] & K \arrow[d] \arrow[r] & H \rtimes F_n \arrow[d] \arrow[r, "\hat p"] & \ZZ \arrow[r] & 1 \\
  1 \arrow[r] & N \arrow[r] & G \arrow[ur, "p"] \\
\end{tikzcd}
	 \]
	 
	 Since $K \rightarrow N$ is onto, and $K$ is finitely generated, $N$ is finitely generated and $G$ algebraically fibers. 

		 \end{proof} 
	
We note that the class of groups $Q$ satisfying the above is large.  For example, free groups, surface groups, and right-angled Artin groups all could be used for $Q$.  
		
	\section{Free of rank 2 by free is incoherent. } \label{sec:F2}

	\begin{theorem} \label{thm:f2fn} 
		Let $G = F_2\rtimes F_n$, where $n \geq 2$. Then $G$ is incoherent. Moreover, $G$ virtually algebraically fibers. 
	\end{theorem}
	\begin{proof}
		Let $F_2$ be the free group generated by $a, b$, let $G = F_2\rtimes F_n$ be a free-by-free group. We have seen that if the natural map $F_n\to \mathrm{Aut}(F_2)$ is not injective, then $G$ contains $F_2\times F_2$ and so is incoherent. Thus we will assume that $F_n\to\mathrm{Aut}(F_2)$ is injective. 
		
		Let $H$ be the subgroup of $F_2$ generated by $x:= a, y:=b^2, z:= bab^{-1}$. This subgroup is not characteristic however it is preserved by an index 3 subgroup of $\mathrm{Aut}(F_2)$. Define two automorphisms of $F_2$ by:  
		\begin{align*}
		\lambda\colon a&\mapsto ab &\rho\colon a&\mapsto a\\
		b&\mapsto b &b&\mapsto ba
		\end{align*}
		
		The automorphisms $\lambda, \rho$ generate $\mathrm{Out}(F_2) = SL_2(\ZZ)$. There are three index two subgroups of $F_2$ and this set is preserved by any automorphism.  Thus we can consider the action of $\lambda$ and $\rho$ on this set and determine the stabilizer of any one of these subgroups in terms of these generators. By inspection, the subgroup $H$ is preserved by $\lambda^2, \rho, \lambda\rho^2\lambda^{-1}$ and $\lambda\rho\lambda\rho^{-1}\lambda^{-1}$. These four elements generate an index 3 subgroup of $\mathrm{Out}(F_2)$. 
		
		Thus we will pass to a finite index subgroup of $G$ which is of the form $G_1 = H\rtimes F$, where $F$ is a free group. The action of $F$ on $H$ is the restriction of the original action (a subgroup of $\Aut(F_2)$)  intersect the index 3 subgroup above. 
		
		We can compute the homology of $G_1$ from  \Cref{lem:homologyoffreebyfree}. So we must compute the matrices corresponding to elements of $F$. We will in fact compute the matrices for the generators of the index 3 subgroup above. 
		We arrive at 
		\begin{align*}
		\Phi_{\lambda^2} &=\left(\begin{matrix}
		1&0&0\\
		1&1&1\\
		0&0&1
		\end{matrix}\right) &\Phi_{\rho} &=\left(\begin{matrix}
		1&1&0\\
		0&1&0\\
		0&1&1
		\end{matrix}\right)\\
		\Phi_{\lambda \rho^2 \lambda^{-1}} &=\left(\begin{matrix}
		0&2&-1\\
		-1&3&-1\\
		-1&2&0
		\end{matrix}\right) &\Phi_{\lambda \rho\lambda\rho^{-1}\lambda^{-1}} &=\left(\begin{matrix}
		2&-1&1\\
		2&-1&2\\
		1&-1&2
		\end{matrix}\right)
		\end{align*}
		
		After taking away the identity matrix from each of the above we can compute the span. The spans of  $\Phi_{\lambda^2} - I$, $\Phi_{\rho} -I$, $\Phi_{\lambda \rho^2 \lambda^{-1}} -I$, and $\Phi_{\lambda \rho\lambda\rho^{-1}\lambda^{-1}} -I$, respectively, are generated by $\begin{bmatrix}
		0 \\
		1 \\
		0\\
		
		\end{bmatrix}$,  $\begin{bmatrix}
		1 \\
		0 \\
		1\\
		
		\end{bmatrix}$, $\begin{bmatrix}
		1 \\
		1 \\
		1\\
		
		\end{bmatrix}$, and  $\begin{bmatrix}
		1 \\
		2 \\
		1\\
		
		\end{bmatrix}$, respectively.  Thus we see that $x$ has infinite order  in the abelianization. 
		
		We can now apply Theorem \ref{thm:incohwithhomology} and see that $G_1$ is incoherent and hence $G$ is incoherent.  We apply Theorem \ref{thm:vf} to see that $G$ virtually algebraically fibers. 
	\end{proof}
	
	If we have a group of the form $G = F_2\rtimes F_n$ where the natural map $F_n\to\mathrm{Out}(F_2)$ is injective, then $G$ embeds in $\mathrm{Aut}(F_2)$. Indeed, $F_2$ is a subgroup of $\Aut(F_2)$. This provides an alternative proof that $\mathrm{Aut}(F_2)$ is incoherent. This was originally proved by Cameron Gordon in \cite{Gordoncoherence}.

	\begin{cor} \cite{Gordoncoherence} 
		$\Aut(F_2)$ is incoherent.
	\end{cor}

	In light of Theorem \ref{thm:f2fn}, we outline a strategy  for how one may attempt to prove all groups of the form $G = F_k\rtimes F_l$ are incoherent. There is a universal group of the form $F_k\rtimes F_n$ coming from any surjection $F_n\to \Out(F_k)$. If one could prove that this group has a finite index subgroup with excessive homology then all subgroups of the form $F_k\rtimes F_l$ will also have excessive homology in a finite index subgroup.  We note that $\Out(F_k)$ can be very different from $\Out(F_2)$ for large $k$. 
	
	\bibliographystyle{plain}

\noindent Robert Kropholler 

\noindent Mathematics M\"unster, WWU M\"unster

\noindent 48149 M\"unster, Germany

\noindent \url{robertkropholler@gmail.com} 
\vskip .2 in 

\noindent Genevieve Walsh 

\noindent Tufts University Mathematics Department 

\noindent Medford, MA USA 02155 	

\noindent \url{genevieve.walsh@tufts.edu} 
	
\end{document}